\documentclass[twoside, 10pt]{article}
\usepackage{amsmath,amsthm,amssymb,amscd,amstext,amsfonts}
\usepackage[utf8]{inputenc}
\usepackage{mathtools}
\usepackage{mathrsfs}
\usepackage{thmtools}
\usepackage{latexsym}
\usepackage{verbatim}
\usepackage{framed}
\usepackage{graphicx}
\usepackage{stmaryrd}
\usepackage{enumitem}
\usepackage{fullpage}
\usepackage{bm}
\usepackage{url}
\usepackage{physics}
\usepackage{dsfont}
\usepackage{todonotes}
\usepackage{tikz}
\usepackage{graphicx}
\usepackage{titling}
\usepackage{color}
\usepackage{xpatch}
\usepackage{algorithm}
\usepackage{algorithmic}
\usepackage{epsf}

\usepackage[small, explicit]{titlesec}
\titlelabel{\thetitle.\enspace}

\usepackage[small, labelfont=bf, labelsep=period]{caption}

\usepackage[breaklinks=false, linktocpage=true,]{hyperref}

\titleformat{\paragraph}[runin]%
{\normalfont\normalsize\bfseries}{\theparagraph}{1em}{#1}[.]
\titlespacing{\paragraph}{0pt}{1.5ex plus 1ex minus .2ex}{0.7em}

\hypersetup{
	colorlinks = true,
    linkcolor={purple},
    urlcolor={purple},
    citecolor={blue!80!black}
}

\makeatletter
\def\tagform@#1{\maketag@@@{(\ignorespaces{\oldstylenums{#1}}\unskip\@@italiccorr)}}
\renewcommand{\eqref}[1]{\textup{{\normalfont(\oldstylenums{\ref{#1}}}\normalfont)}}
\newcommand{\ps@bookheader}{%
\renewcommand\@oddfoot{\hfil}%
\renewcommand\@evenfoot{\hfil}%
\renewcommand\@oddhead{\ifnum\value{page}>1
{\small\hfil BEKE, GOH, HATAMI, JAFFE, AND NAYLOR}\hfil\thepage\else\hfil\fi}}
\renewcommand\@evenhead{\thepage\hfil\small A CHARACTERIZATION OF IDEMPOTENT SCHUR MULTIPLIERS\hfil}%
\makeatother
\usepackage[letterpaper, total={5.5in, 8in}, vmarginratio=1:1, hmarginratio=1:1, headsep=21pt]{geometry}

\usepackage[nameinlink, noabbrev, capitalize]{cleveref}

\declaretheoremstyle[bodyfont=\normalfont\slshape, notefont=\normalfont\itshape,notebraces={{\rm(}}{{\rm)}}, postheadspace=0.5em,headpunct={\rm.}, spaceabove=8pt, spacebelow=8pt]{slbody}%
\declaretheoremstyle[bodyfont=\normalfont\slshape, notefont=\normalfont\itshape,notebraces={{\rm(}}{{\rm)}}, postheadspace=0.5em,headpunct={.}, spaceabove=8pt, spacebelow=8pt]{slbodynoparen}

\declaretheorem[name=Theorem, numberwithin=section, style=slbody]{theorem}
\declaretheorem[name=Lemma, numberwithin=section, sibling=theorem, style=slbody]{lemma}

\declaretheorem[name=Proposition, numberwithin=section, sibling=theorem, style=slbody]{proposition}

\declaretheorem[name=Theorem, numberwithin=section, sibling=theorem, style=slbodynoparen]{theoremnoparen}
\declaretheorem[name=Lemma, numberwithin=section, sibling=theorem, style=slbodynoparen]{lemmanoparen}

\declaretheorem[name=Proposition, numberwithin=section, sibling=theorem, style=slbodynoparen]{propositionnoparen}

\declaretheorem[name=Claim, numberwithin=section, sibling=theorem, style=slbodynoparen]{claimnoparen}

\renewcommand\norm[1]{\left|\!\left|#1\right|\!\right|}

\newcommand\bignorm[1]{\bigl|\!\bigl|#1\bigr|\!\bigr|}
\newcommand\normmax[1]{|\!|#1|\!|_{\rm max}}

\newcommand\normrow[1]{\left|\!\left|#1\right|\!\right|_{\rm row}}
\newcommand\normcol[1]{\left|\!\left|#1\right|\!\right|_{\rm col}}
\newcommand\normgamma[1]{|\!|#1|\!|_{\gamma_2}}

\newcommand\normalg[1]{\left|\!\left|#1\right|\!\right|_{\rm A}}
\newcommand\normmul[1]{\left|\!\left|#1\right|\!\right|_{\rm m}}

\newcommand\eps{\epsilon}

\DeclareMathOperator{\supp}{supp}

\DeclareMathOperator{\block}{block}

\newcommand{\RR}{\mathbf{R}}   
\newcommand{\CC}{\mathbf{C}}   
\newcommand{\ZZ}{\mathbf{Z}}   

\newcommand{\one}{\mathop{\mathbf{1}}\nolimits}

\newcommand\ex{\mathop{\mathbf{E}}\nolimits}

\renewcommand{\maketitle}{%
  \begin{center}
    {\large\bf A characterization of idempotent Schur multipliers}\\
    \vskip 30pt
    {\sc Csongor Beke, Marcel K. Goh, Hamed Hatami,}
    \\
    \medskip
    {\sc Sean Jaffe, {\rm and} Daniel Naylor}
    \medskip
    \vskip 30pt
  \end{center}
}
\title{}\author{}\date{}

\begin{document}

\maketitle 
 
\renewenvironment{abstract}{\quotation\noindent\small{\bfseries\abstractname.\enspace}}{\endquotation}

\begin{abstract}
We prove that every idempotent Schur multiplier is a finite signed sum of contractive idempotent Schur multipliers. This was conjectured by Katavolos and Paulsen in 2003 and previously known only for translation-invariant Schur multipliers, by the Cohen--Host idempotent theorem.

Concretely, we show that any boolean matrix $A$ with Schur multiplier norm at most~$\gamma$ (or equivalently $\lVert A\rVert_{\gamma_2} \le \gamma$) can be written as
\[ A=\sum_{i=1}^{L}\sigma_i B_i,\]
where $L\leq 2^{C\gamma^6}$ for an absolute constant $C$, $\sigma_i\in\{-1,1\}$ are signs, and each $B_i$ is a contractive idempotent Schur multiplier, that is, a boolean matrix whose $1$-entries form a union of all-one rectangular blocks, with no two blocks sharing a row or a column.

As observed by Carenini, a key lemma in our work yields a new proof of the Cohen--Host theorem and gives a simple proof of the quantitative refinements of Green--Sanders and Sanders, with improved bounds. We include a self-contained exposition of these results in the case of finite groups.
\end{abstract}

\vskip30pt
\baselineskip=13.5pt

\section{Introduction} 

The theory of idempotents of the Fourier--Stieltjes algebra $B(G)$ of a locally compact group $G$ is one of the classical subjects in abstract harmonic analysis. Since an idempotent $f\in B(G)$ satisfies $f^2=f$ pointwise, it is necessarily of the form $f=1_A$ for some subset $A\subseteq G$. The simplest case is that of \emph{contractive} idempotents: every nonzero contractive idempotent is the indicator of a coset of an open subgroup~\cite{MR2152508}.  This phenomenon is closely related to the classical theorem of Kawada and It\^{o}~\cite{KawadaIto1940} characterizing idempotent probability measures on compact groups.

A recurring principle in the theory is that general idempotents are built from contractive idempotents. Helson~\cite{helson1953} first established this phenomenon for $G=\ZZ$, and Rudin~\cite{rudin1959} subsequently extended it to $G=\ZZ^d$.  Building on this work, Cohen proved his celebrated idempotent theorem~\cite{cohen1960}: \textsl{For every locally compact abelian group $G$, the indicator $1_A$ belongs to $B(G)$ if and only if $A$ belongs to the \emph{coset ring} of $G$, that is, the boolean algebra generated by cosets of open subgroups.} Equivalently, every idempotent in $B(G)$ can be expressed as a finite signed sum of indicators of such cosets. Host~\cite{host1986} later extended Cohen's theorem to arbitrary locally compact groups.

Decades later, following modern developments in additive combinatorics, Green and Sanders~\cite{greensandersannals} in the abelian case, and Sanders~\cite{sanders2011} for finite non-abelian groups, obtained quantitative refinements of Cohen--Host theorems in which the number of cosets is bounded solely in terms of the Fourier--Stieltjes norm of the idempotent.  

The starting point of this paper is that the same picture has a natural analogue for Schur multipliers. As we explain below, every $f\in B(G)$ gives rise to a \emph{translation-invariant} Schur multiplier, and the indicators of open cosets correspond to contractive idempotent Schur multipliers. Thus, from the Schur multiplier viewpoint, the Cohen--Host theorem says that every translation-invariant idempotent arising from $B(G)$ is a finite signed sum of contractive idempotents.

This naturally suggests carrying the same program beyond the translation-invariant setting to the full algebra of Schur multipliers: \textsl{Is every idempotent Schur multiplier a finite signed sum of contractive Schur multipliers?} We answer this question affirmatively and supply a quantitative bound as well. To state our result precisely, we first recall a number of relevant definitions.

\medskip 
Let $X$ and $Y$ be finite or countable sets. We identify each bounded operator
$T:\ell_2(Y)\to\ell_2(X)$ with its matrix indexed by $X\times Y$, and write
$\norm{T}$ for its operator norm. For a complex-valued matrix
$A:X\times Y\to\CC$, let $A\circ T$ denote the entrywise (Schur) product. The \emph{Schur multiplier norm} of $A$ is defined  by
\[ \normmul{A}= \sup_{0 < \norm T < \infty} \frac{\norm{A\circ T}}{\norm T},\]
where the value is understood to be infinite if $A\circ T$ does not define
a bounded operator for some bounded $T$. We call $A$ a \emph{Schur
multiplier} whenever $\normmul{A}<\infty$.  

The Schur multipliers form a commutative algebra under addition and Schur product. A Schur multiplier $A$ is an \emph{idempotent} if $A\circ A=A$, and is \emph{contractive} if $\normmul{A}\leq 1$. Since idempotence holds entrywise, the idempotent Schur multipliers are precisely the boolean matrices with finite Schur multiplier norm.

\begin{theorem}\label{thmmain}
There exists an absolute constant $C$ such that every idempotent Schur multiplier $A$ with $\normmul{A}\leq \gamma$ can be written as
\[ A=\sum_{i=1}^L \sigma_i B_i, \]
where each $B_i$ is a contractive idempotent Schur multiplier, $\sigma_i\in\{-1,1\}$, and $L\leq 2^{C\gamma^4}$.
\end{theorem} 

While \Cref{thmmain} does not by itself recover the classical idempotent theorems in harmonic analysis, subsequent work of Carenini~\cite{carenini2026} shows that our key lemma can be applied equivariantly to obtain a symmetry-preserving version of \Cref{thmmain}. The resulting theorem recovers the Cohen--Host theorem and its quantitative refinements, with considerably simpler proofs that use no tools from additive combinatorics. It also yields substantially improved quantitative bounds in the non-abelian case:~the triple exponential bound of Sanders~\cite{sanders2021coset} is replaced by a single-exponential bound.
 
An earlier version of this paper gave an $O(\gamma^6)$ exponent in \Cref{thmmain}. Carenini subsequently obtained $O(\gamma^4)$ using a martingale argument, after which we observed that minor cosmetic modifications to our proof yields the same bound. 

\Cref{thmmain} extends easily to Schur multipliers over measure spaces. More precisely, if $(X,\mu)$ and $(Y,\nu)$ are standard $\sigma$-finite measure spaces and $P\in L^\infty(X\times Y)$ is an idempotent measurable Schur multiplier with $\normmul P\leq\gamma$, then $P$ can be expressed as a signed sum of at most $2^{C\gamma^4}$ contractive measurable Schur multipliers.  We discuss this extension in detail in \Cref{sec:measurable}.

\paragraph{The qualitative problem and previous progress}
For an infinite-dimensional boolean matrix $A$, \Cref{thmmain} shows that $\normmul A<\infty$ if and only if $A$ is a finite signed sum of blocky matrices; equivalently, every idempotent Schur multiplier is a finite signed sum of contractive idempotents. Whether this holds was, to our knowledge, first asked by Katavolos and Paulsen~\cite{kp05}. By 2010, it had reportedly become ``one of the difficult open problems in the area''~\cite{MR2777487}, and it has repeatedly been highlighted since; see, for instance, \cite{todorov2015} and \cite[Question~3.13]{elt16}.

We are not aware of any significant progress toward the problem until 2025, when Balla, Hambardzumyan, and Tomon~\cite{BHTpublished26} showed that every finite-dimensional boolean matrix of bounded Schur multiplier norm contains a monochromatic submatrix of constant density. Subsequently, Goh and Hatami~\cite{blockysubset2025} proved that the support of such a matrix contains a blocky matrix covering a significant portion of its $1$-entries. In later work~\cite{blockypolylog2025}, the same authors showed that every $n\times n$ idempotent Schur multiplier $A$ admits a signed decomposition into at most $2^{O(\normmul{A}^7)}(\ln n)^2$ contractive idempotent Schur multipliers. \Cref{thmmain} fully resolves the conjecture by building on this work. 
 
\paragraph{Contractive idempotents and blocky matrices}
To put \Cref{thmmain} in a more concrete form, we next describe the structure of its building blocks. The contractive idempotent Schur multipliers admit a particularly simple characterization.

A boolean matrix $B:X \times Y \to \{0,1\}$ is called \emph{blocky} if its support is exactly $\bigcup_{i \in I} X_i \times Y_i$ for some families $\{X_i\}_{i\in I}$ and $\{Y_i\}_{i\in I}$ that are each pairwise disjoint. Livshits~\cite{liv95} showed that blocky matrices are exactly the \emph{contractive} idempotents, namely those with $\normmul{A} \le 1$.  

\begin{proposition}[{\rm\cite{liv95}}]
\label{prop:liv}
A nonzero boolean matrix $A:X\times Y\to\{0,1\}$ satisfies
$\normmul{A}\le 1$ if and only if $A$ is blocky, and in that case
$\normmul{A}=1$.
\end{proposition}

\begin{proof}
Suppose first that $A$ is blocky, so that the support of $A$ is $\bigcup_{i\in I} X_i\times Y_i$,
where the families $\{X_i\}_{i\in I}$ and $\{Y_i\}_{i\in I}$ are each
pairwise disjoint. Then for every bounded operator $T$, the matrix
$A\circ T$ is the direct sum of the submatrices $T_{X_i\times Y_i}$.
Hence
\[
\norm{A\circ T}
=
\sup_{i\in I}\norm{T_{X_i\times Y_i}}
\le \norm{T},
\]
and therefore $\normmul{A}\le 1$. Since every nonzero boolean matrix
satisfies $\normmul{A}\ge 1$, we in fact have $\normmul{A}=1$.

Conversely, a boolean matrix is blocky if and only if it contains no
$2\times 2$ submatrix with exactly three $1$-entries. Since the Schur
multiplier norm does not increase upon passing to a submatrix, and
\[
\normmul{
\begin{pmatrix}
1&0\\
1&1
\end{pmatrix}
}
=
\frac{2}{\sqrt{3}}
>
1,
\]
the presence of such a submatrix would imply $\normmul{A}>1$.
Thus every boolean matrix with $\normmul{A}\le 1$ is blocky.
\end{proof}

Thus the contractive idempotent Schur multipliers are exactly the blocky matrices. We now introduce a complexity measure for decompositions of the type appearing in \Cref{thmmain}. We define the \emph{block complexity} $\block(A)$ of a matrix $A$ to be the least integer $L$ for which there exist blocky matrices $B_1,\ldots,B_L$ and signs $\sigma_1,\ldots,\sigma_L\in\{-1,1\}$ such that
\[
A=\sum_{i=1}^L\sigma_iB_i,
\]
with $\block(A)=\infty$ if no such representation exists. In this notation, \Cref{thmmain} states
$\block(A)\leq 2^{C\normmul{A}^4}$, while the triangle inequality gives $\normmul{A}\leq\block(A)$.

\paragraph{Connection to communication complexity}
The problem studied here also has a natural formulation in communication complexity. 
Blocky matrices are precisely the matrices corresponding to queries to an equality oracle, 
and decompositions into blocky matrices are closely related to deterministic communication 
protocols with access to such an oracle. This led to the conjecture that a family of boolean matrices has uniformly bounded equality-oracle complexity if and only if its Schur multiplier norms are uniformly bounded; see~\cite[Section~3]{hatami2024guest} for a detailed discussion. \Cref{thmmain} resolves this conjecture. 

\paragraph{A note on notation} Expressions such as $f = O(g)$ always hide absolute constants. We write $A \in S^{X\times Y}$ to mean that $A$ is a matrix with rows indexed by elements of $X$, columns indexed by elements of $Y$, and entries taking values in the set $S$. For subsets $X'\subseteq X$ and $Y'\subseteq Y$, the expression $A_{X'\times Y'}$ denotes the matrix $A$ restricted to rows in $X'$ and columns in $Y'$.

\paragraph{Statement regarding artificial intelligence}
None of the mathematical arguments or writing in this paper was generated by a large language model.

\section{The key lemma}

The main new ingredient in our proof of \Cref{thmmain} is a lemma that gives us a norm-reducing decomposition of some column of $A$. In this brief section we shall simply state this key lemma, but in order to do so we must first establish some definitions and notation.

For the rest of the paper, rather than working directly with the Schur multiplier norm, we use the following characterization that is often attributed to Grothendieck~\cite{grothendieck}. The \emph{$\gamma_2$ factorization norm} (or simply the \emph{$\gamma_2$ norm}) of a real matrix $A$ is defined by
\begin{equation}
\normgamma{A}
=
\min_{UV=A}\normrow{U}\normcol{V},
\end{equation}
where the minimum is taken over all factorizations $A=UV$, $\normrow{U}$ denotes the maximum $\ell_2$-norm of a row of $U$, and $\normcol{V}$ denotes the maximum $\ell_2$-norm of a column of $V$.
A classical result of Grothendieck~\cite{grothendieck} shows that $\normmul{A}=\normgamma{A}$ for every matrix $A$.

The $\gamma_2$ norm does not increase upon restricting to a submatrix, since deleting rows of $U$ or columns of $V$ cannot increase $\normrow{U}$ or $\normcol{V}$. In particular,
\[
\normgamma{A}\geq \normmax{A}
=
\max_{(x,y)\in X\times Y}|A(x,y)|.
\]

For a complex-valued matrix $A:X\times Y \to \CC$ with $\normgamma A\le \gamma$, we define a \emph{$\gamma$-factoriza\-tion} of $A$ to be a factorization $A = UV$ where $\normrow{U}\le 1$ and $\normcol V\le \gamma$. If $A$ is real-valued, then $U$ and $V$ can be assumed to be real-valued as well. In such a factorization we will denote the rows of $U$ by $u_x$, and the columns of $V$ by $v_y$. In particular, we will often write $A(x, y) = \langle u_x, v_y \rangle$.

We note that \Cref{lemdecomposition} is similar to the decomposition lemma used for the inductive step in~\cite{greensandersannals}. Furthermore, both decompositions are proved using a similar technique: one part of the decomposition is constructed by some form of averaging, and the other part is constructed as the difference from that average.

As in~\cite{greensandersannals}, in order to work with averages, we will need to work with almost integer matrices.
Given a matrix $A\in \RR^{X\times Y}$, let $A_\ZZ\in \ZZ^{X\times Y}$ denote the entrywise rounding of $A$ to the nearest integer matrix, where we round $m+1/2$ down to $m$ for all $m\in \ZZ$. For any parameter $\eps > 0$, we say that a real-valued matrix $A\in \RR^{X\times Y}$ is \emph{$\eps$-almost integer-valued} if $\normmax{A-A_\ZZ}\le \eps$.

For an integer-valued matrix $A\in \ZZ^{X\times Y}$ we define 
\[D(A) = \max_{x \in X} \sum_{y\in Y} \bigl| A(x,y)\bigr| .\]
Our key lemma presents a dichotomy:~either $D(A_\ZZ)$ is small or there exists a vector in the $\gamma$-factorisation of $A$ that admits an almost integer-valued norm-reducing decomposition.

\begin{lemma}[Key lemma]
\label{lemdecomposition}
Let $\gamma > 1/2$ and $0<\eps< 1/6$. Let $A\in \RR^{X\times Y}$
be a finite $\eps$-almost integer-valued matrix with $\normgamma A\le \gamma$, and let $A=UV$ be a $\gamma$-factorization. Suppose that $A_\ZZ$ has no duplicate columns. Then either
\[D(A_\ZZ) < 4(\gamma+1)(6\gamma / \eps)^{16\gamma^2}\]
or there exists $y_0 \in Y$ and a decomposition $v_{y_0} = v' + v''$ with $\norm{v'}^2 + \norm{v''}^2 \le \norm{v_{y_0}}^2$ such that
\begin{itemize}
\item[i)] $\max\bigl(\norm{v'}^2, \norm{v''}^2\bigr) \le \gamma^2 - 1/16$; and
\item[ii)] for every $x \in X$, $\langle u_x, v' \rangle$ and $\langle u_x, v'' \rangle$ are at distance 
\end{itemize}
\end{lemma}

The bound $\norm{v'}^2+\norm{v''}^2\leq\norm{v}^2$ is not needed for our proof of \Cref{thmmain}, but we record it since it comes easily from the argument, and because it is useful for getting better bounds in other applications (for example, when bounding the number of subgroups in the Green--Sanders idempotent theorem).

We postpone the proof of \Cref{lemdecomposition} to \Cref{sec:proofofdecomposition}. In the next section, \Cref{sec:mainthm}, we apply \Cref{lemdecomposition} to deduce \Cref{thmmain}, and in \Cref{sec:groupidempotent} we discuss an application of \Cref{lemdecomposition} due to Carenini~\cite{carenini2026} that gives a new proof of the quantitative non-abelian idempotent theorem of Sanders \cite{sanders2011, sanders2021coset}. 

\section{Proof of the main theorem}\label{sec:mainthm}

We deduce \Cref{thmmain} from \Cref{lemdecomposition}.
When $D(A)$ is small, a greedy algorithm from~\cite{blockypolylog2025} can be used to find an efficient blocky decomposition.

\begin{lemma}[{\rm\cite{blockypolylog2025}}, Lemma 4.1]\label{lemblockcomplexitybound}
Every integer-valued matrix $A\in \ZZ^{X\times Y}$ satisfies
\begin{equation}\block(A) \le 2D(A).\end{equation}
\end{lemma}

\begin{proof}
Consider the decomposition $A = A_+ - A_-$, where $A_+$ and $A_-$ have non-negative entries. Since $D(A_+), D(A_-) \le D(A)$ and $\block(A) \le \block(A_+) + \block(A_-)$, it suffices to show $\block(A) \le D(A)$ for non-negative integer matrices $A$.
    
To see this, greedily decompose $A$ into $A_1 + A_2 + \cdots + A_{D(A)}$, with each $A_i$ being a boolean matrix with at most one nonzero entry per row. Any boolean matrix with at most one nonzero entry per row is automatically blocky, so this shows $\block(A) \le D(A)$.
\end{proof}

If, on the other hand, $D(A)$ is large, then we use \Cref{lemdecomposition} to repeatedly remove columns until the remaining matrix $R$ has small $D(R)$.

\begin{lemmanoparen}\label{lemkey} Let $\gamma > 1/2$ and $0< \eps < 1/32$.
Let $A\in \RR^{X\times Y}$ be a finite $\eps$-almost integer-valued matrix with $\normgamma A \le \gamma$ such that $A_\ZZ$ has no repeated columns. Then there is a subset $Y' \subseteq Y$ and $3\eps$-almost integer-valued matrices $A',A''\in \RR^{X\times Y'}$ such that
\begin{itemize}
\item[i)] $A'_\ZZ + A''_\ZZ = (A_{X \times Y'})_\ZZ$;
\item[ii)] $\max\bigl(\normgamma{A'}^2, \normgamma{A''}^2\bigr)
\le \gamma^2 - 1/16$; and
\item[iii)] $D((A_{X \times (Y \setminus Y')})_\ZZ) \le 4(\gamma+1)(6\gamma/\eps)^{16\gamma^2}$.
\end{itemize}
\end{lemmanoparen}

\begin{proof}
Fix a $\gamma$-factorisation $A = UV$. We inductively define sequences $(Y_i)$, $(V_i')$, and $(V_i'')$, where $V_i'$ and $V_i''$ are matrices satisfying $V_i' + V_i'' = V|_{Y_i}$. In particular, the columns of $V_i'$ and $V_i''$ are indexed by $Y_i$.

We start with $Y_1$, $V_1'$ and $V_1''$ all empty. If $D(A_{X \times (Y \setminus Y_i)}) \ge 4(\gamma+1)(6\gamma/\eps)^{16\gamma^2}$, apply \Cref{lemdecomposition} to $A_{X \times (Y \setminus Y_i)}$ to get some $y_{i+1} \notin Y_i$ and a decomposition $v_{y_{i+1}} = v' + v''$. We now add $y_{i + 1}$ to $Y_i$ to form $Y_{i + 1}$, and similarly form $V_{i+1}'$, $V_{i+1}''$ by inserting the vectors $v'$, $v''$ respectively.

Since $|Y_{i+1}|$ strictly decreases at each step, this process eventually terminates, say at some $Y_t$, $V_t'$, and $V_t''$. We now set $Y' =Y_t$, $A' = UV_t'$ and $A'' = UV_t''$. The properties of the decomposition given by \Cref{lemdecomposition}, combined with the fact that $3\eps \le 1/8$, give claims (i) and (ii). Claim (iii) follows from the condition for terminating the sequence.
\end{proof}

We are now ready to prove \Cref{thmmain}. As mentioned earlier, we actually prove a stronger statement regarding almost-integer matrices.

\begin{theorem}[Generalized main theorem]\label{thmgeneralmain}
Let $\gamma > 0$ and $\eps = 2^{-32\gamma^2-5}$, and suppose that $A\in \RR^{X\times Y}$ is an $\eps$-almost integer-valued matrix with $\normgamma A \le \gamma$. Then
\begin{equation}\block(A_\ZZ) \le 2^{1000\gamma^{4}}.\end{equation}
\end{theorem}

\begin{proof}
The proof is by induction on $\gamma$. Since $\gamma < 1/2$ forces $A_\ZZ = 0$, the base case is trivial. A straightforward compactness argument, recorded in~\cite[Theorem 3.10]{hhh2023}, allows us to assume without loss of generality that $A$ is a finite-dimensional matrix. Since collapsing repeated columns of $A_\ZZ$ to a single column does not change $\block(A_\ZZ)$, we may also assume that $A_\ZZ$ has no repeated columns. 

We can now apply \Cref{lemkey}. After adding some all-zero columns, this gives
\[ A_\ZZ = A'_\ZZ + A''_\ZZ + R_\ZZ ,\]
where $A'$ and $A''$ are $3\eps$-almost integer-valued, $\normgamma{A'}^2$ and $\normgamma{A''}^2$ are at most $\gamma^2 - 1/16$, and, by \Cref{lemblockcomplexitybound},
\[ \block(R_\ZZ) \le 2 \cdot 4(\gamma + 1)(6\gamma/\eps)^{16\gamma^2} \le 2^{999\gamma^4} .\]
Since $3\eps \le 2^{-32(\gamma^2 - 1/16) - 5}$, we can apply the inductive hypothesis to $A'$ and $A''$. Thus
\[ \block(A_\ZZ) \le 2^{1000(\gamma^2 - 1/16)^2} + 2^{1000(\gamma^2 - 1/16)^2} + 2^{999\gamma^4} \le 2^{1000\gamma^4} . \qedhere \]
\end{proof}

\section{Carenini's application to the Fourier--Stieltjes algebra}\label{sec:groupidempotent}

In this section we discuss a special case of a result due to Carenini~\cite{carenini2026} that extends \Cref{thmmain} to an equivariant setting. The special case that we discuss is the following quantitative improvement of Sanders's idempotent theorem for finite groups~\cite{sanders2020}. 

\begin{theoremnoparen}[{\rm\cite{carenini2026}}]\label{theoremnonabelianGS}
    Let $\gamma > 0$ and $\eps = 2^{-32\gamma^2 - 5}$. Let $G$ be a finite group and let $f : G \to \RR$ be $\eps$-almost integer-valued with $\normalg{f} \le \gamma$. Then there exists $L \le \exp(O(\gamma^4))$, signs $\sigma_i \in \{-1, 1\}$, subgroups $K_i \le G$ and elements $a_i \in G$ such that
    \[ f_\ZZ = \sum_{i = 1}^L \sigma_i \one_{a_i K_i}.\]
    Furthermore, the number of distinct subgroups in $K_1, \ldots, K_L$ is at most $4\gamma^2$.
\end{theoremnoparen}

The bound on the number of subgroups here is essentially optimal for the non-abelian case, and can easily be improved to $\gamma^2 + o(1)$\footnote{A paper of Sanders \cite{sanders2021coset} claims a bound of $O(\gamma)$ on the number of subgroups used, but this arises from a calculation error in the proof of \cite[Lemma 2.1]{sanders2021coset}. After replacing \cite[Lemma 2.1]{sanders2021coset} with \Cref{lemmanormineqsimple}, the argument gets the same bound of $\gamma^2 + o(1)$ on the number of subgroups.}, after adjusting $L$ and $\eps$ slightly. For completeness, and to illustrate the role of \Cref{lemdecomposition}, we include a sketch of Carenini's proof. 

\begin{proof}[Proof sketch of \Cref{theoremnonabelianGS}]
    The proof is by induction on $\gamma$. Since $\gamma < 1/2$ forces $f_\ZZ = 0$, the base case is trivial.
    Define the right-period subgroup
    \[ H = \bigl\{h \in G : \,f_\ZZ(xh) = f_\ZZ(x)\ \hbox{for all}\ x \in G \bigr\} ,\]
    and let $G / H$ be the set of left cosets of $H$.
    
    By Eymard's definition of $\normalg{f}$ (see \cite[Definition 2.2]{eymard64}), there exist $u, v \in \ell^2$ and an orthogonal representation $\rho : G \to O(\ell^2)$ such that $\norm{u} \le 1$, $\norm{v} \le \gamma$, and for all $x \in G$,
    \[ f(x) = \bigl\langle \rho(x) u, v \bigr\rangle .\]
    Define the matrix $M \in \RR^{G \times G}$ by $M(x, y) = f(y^{-1}x)$. We will use the $\gamma$-factorisation $M(x, y) = \langle \rho(x) u, \rho(y) v \rangle$.
    
    Arbitrarily choose coset representatives $b_1 H, \ldots, b_n H$ of $G / H$, and define the matrix $M / H$ by $(M / H)(x, b_i H) = M(x, b_i)$. We now apply \Cref{lemdecomposition} to $M / H$.

    If $D((M / H)_\ZZ) \le 2^{O(\gamma^4)}$, then the function $f_\ZZ / H : G / H \to \RR$ induced by $f_\ZZ$ satisfies $\|f_\ZZ / H\|_1 \le 2^{O(\gamma^4)}$, since $(M / H)_\ZZ(e_G, a_i H) = (f_\ZZ / H)(a_i^{-1}H)$. Thus there exists $L \le 2^{O(\gamma^4)}$, $a_1, \ldots, a_L \in G$, and $\sigma_1, \ldots \sigma_L \in \{-1, 1\}$ such that
    \[ f_\ZZ = \sum_{i = 1}^L \sigma_i \one_{a_i H} ,\]
    which gives a decomposition of the desired form.

    Otherwise, \Cref{lemdecomposition} gives a decomposition $\rho(y_0) v = v' + v''$. In this case, we decompose $f$ by defining $f'(x) = \langle \rho(x) u, \rho(y_0^{-1}) v' \rangle$ and $f''$ similarly. A straightforward calculation shows $f = f' + f''$. Property (i) of \Cref{lemdecomposition} gives $\max(\normalg{f'}^2, \normalg{f''}^2) \le \gamma^2 - 1/16$, and property (ii) shows that $f'$ and $f''$ are $3\eps$-almost integer-valued. The result now follows by applying the induction hypothesis to $f'$ and $f''$.
\end{proof}

The above application of \Cref{lemdecomposition} gives a decomposition $f = f * \mu_S + f''$, where $\mu_S$ is the uniform probability measure on a subset $S \subseteq G$, with $S$ chosen such that $f * \mu$ is almost integer-valued. To show that the decomposition is norm-reducing, the above proof makes use of the following fact.

\begin{lemmanoparen}\label{lemmanormineqsimple}
    Let $f : G \to \RR$ and $S \subseteq G$. Then there exists $s \in S$ such that
    \[ \normalg{f * \mu_{S - s}}^2 + \normalg{f - f * \mu_{S - s}}^2 \le \normalg{f}^2 .\]
\end{lemmanoparen} 

Green and Sanders~\cite{greensandersannals} use a very similar type of decomposition: they decompose $f = f * \mu_S * \mu_{-S} + f''$, where $S \subseteq G$ is chosen so that the decomposition is almost integer-valued. Since the Fourier transform of $\mu_S * \mu_{-S}$ takes values in $[0, 1]$, it follows that $\normalg{f} = \normalg{f'} + \normalg{f''}$, and this is used to show that the decomposition is norm-reducing.

In~\cite{greensandersannals}, $S$ is constructed from a regular Bourgain system, and in other related works by Green and Sanders, $S$ is taken to be either a subgroup or an approximate subgroup. On the other hand, in the above proof, the set $S$ is constructed by a simple fact about vectors in Euclidean space (see \Cref{prop:etadeltaconstantsubset}).

\section{Proof of the key lemma}\label{sec:proofofdecomposition}

We now prove \Cref{lemdecomposition}.  The result is immediate for matrices $A$ with
\[D(A_\ZZ) < 4(\gamma+1)(6\gamma / \eps)^{16\gamma^2},\]
so we assume the opposite holds and begin with the following observation.

\begin{claimnoparen}\label{obsYprime}
    There exists a unit vector $u$ and a subset $Y' \subseteq Y$ of size at least $2(6\gamma/\eps)^{16\gamma^2}$ such that $\langle u_0, v_y \rangle \ge 1/2$ for every $y \in Y'$.
\end{claimnoparen} 

\begin{proof}
    Let $x$ be any row such that
    \[ \sum_{y \in Y} |A_\ZZ(x, y)| = D(A_\ZZ) \ge 4(\gamma + 1)(6\gamma/\eps)^{16\gamma^2} .\]
    Since $\normmax{A_\ZZ} \le \gamma + 1$, $A_\ZZ$ has at least $4(6\gamma/\eps)^{16\gamma^2}$ nonzero entries on row $x$. Without loss of generality, by replacing $A$ by $-A$ if necessary, we may assume that at least half of the nonzero entries of $A_\ZZ$ on this row are positive. Let $u_0 = u_x / \norm{u_x}$ and let $Y'$ be the set of columns that satisfy $A_\ZZ(x, y) \ge 1$. Then $|Y'| \ge 2(6\gamma/\eps)^{16\gamma^2}$, and for all $y \in Y'$ we have
    \[ \langle u, v_y \rangle \ge \langle u_x, v_y \rangle = A(x, y) \ge 1 - \eps \ge \frac{1}{2} . \qedhere \]
\end{proof}

Our strategy for constructing the decomposition will be to find $S \subseteq Y'$ with $|S| \ge 2$ such that $v' = \ex_{y \in S} v_y$ has almost-integer inner products with every $u_x$. We observe that defining $v'$ in this way gives $1/16 \le \norm{v'}^2 \le \gamma^2 - 1/16$.

\begin{claimnoparen}\label{obsvprime}
    Let $S \subseteq Y'$ satisfy $|S| \ge 2$. Then $v' = \ex_{y \in S} v_y$ satisfies
    \[ \frac{1}{16} \le \norm{v'}^2 \le \gamma^2 - \frac{1}{16} .\]
\end{claimnoparen}

\begin{proof}
    For the lower bound, note that $\langle u_0, v' \rangle = \ex_{y \in S} \langle u_0, v_y \rangle \ge \frac{1}{2}$ by the definition of $Y'$. By Cauchy--Schwarz, this gives $\norm{v'} \ge 1/2$ and hence $\norm{v'}^2 > 1/16$.

    For the upper bound, note that for every $y_1 \neq y_2 \in S$ there exists $x \in X$ such that $A(x, y_1)$ and $A(x, y_2)$ round to different integers (since $A_\ZZ$ has no repeated columns). Since $A$ is $\eps$-almost integer-valued, this gives $|\langle u_x, v_{y_1} - v_{y_2} \rangle| \ge 1/2$. By Cauchy--Schwarz, this gives $\norm{v_{y_1} - v_{y_2}} \ge 1/2$, and thus the parallelogram identity gives
    \[ \norm{v_{y_1} + v_{y_2}}^2 = 2\norm{v_{y_1}}^2 + 2\norm{v_{y_2}}^2 - \norm{v_{y_1} - v_{y_2}}^2 \le 4\gamma^2 - \frac{1}{4} .\]
    Using the fact that $|S| \ge 2$, the upper bound now follows from
    \[v' = \ex_{y_1 \neq y_2 \in S} \frac{v_{y_1} + v_{y_2}}{2}.\qedhere\]
\end{proof}

Next, we choose $S \subseteq Y'$. The idea is to choose $S$ such that $A$ is almost constant on $S$, since taking an average of an almost constant set of integers gives an almost integer. This idea comes from the work of Green and Sanders~\cite{greensanders,greensandersannals}, and had previously been applied to Schur multipliers in \cite{blockypolylog2025}.

Let $V\subseteq \RR^n$ be a finite set of vectors. We say that $V$ is \emph{$(\eta, \delta)$-constant} if for every unit vector $u \in \RR^n$ there exists $g \in \RR$ such that
\[ \Pr_{v \in V}(|\langle u, v\rangle - g| > \delta) \le \eta .\]
The following result can be deduced from Propositions 3.1 and 3.2 in \cite{blockypolylog2025}, but gives a slightly worse bound. The argument there is inspired by the classic mistake-bound analysis \cite{rosenblatt1958} of the perceptron algorithm \cite{mp43}.

We include a proof here with some cosmetic changes and a tighter bound. The proof in \cite[Proposition 3.2]{blockypolylog2025} gives a bound of $O(\gamma^4)$ in the exponent, and it was first observed by Carenini~\cite{carenini2026} that this can be improved to $O(\gamma^2)$. We remark that this improved bound can also be obtained from the proof in~\cite{blockypolylog2025} by replacing the $\gamma$-factorisation convention to $\norm{u_x} \le \gamma$, $\norm{v_y} \le 1$.

\begin{propositionnoparen}\label{prop:etadeltaconstantsubset}
Let $\eta \in (0, 1]$, $\gamma, \delta > 0$ and let $V \subseteq \RR^n$ be a finite set of vectors of length at most $\gamma$. Then there exists $S \subseteq V$ which is $(\eta, \delta)$-constant and satisfies
\[ |S| \ge (\eta/2)^{\gamma^2/\delta^2}|V| .\]
\end{propositionnoparen}

\begin{proof}
Let $C$ be the convex hull of $V$. We prove the slightly stronger bound
\begin{equation}
|S|\geq (\eta/2)^{(\gamma^2-d(0,C)^2)/\delta^2}|V|
\label{eq:littlestoneinductivehypothesis}
\end{equation}
by induction on  $d(0, C)$. 

If $V$ is $(\eta,\delta)$-constant, there is nothing to prove as we simply take $S = V$. Otherwise, let $u$ be a unit vector witnessing the contrary. Put
$m=\lceil \eta|V|/2\rceil$, and let $V^+$ and $V^-$ consist of the $m$ vectors with respectively the largest and smallest values of $\langle u,v\rangle$.

Since $u$ witnesses that $V$ is not $(\eta,\delta)$-constant, $V^+$ and $V^-$ are separated by more than $2\delta$ in the $u$-direction. Hence, if $C^+$ and $C^-$ denote their convex hulls,
\[
d(C^+,C^-)>2\delta.
\]
Pick $z\in C$ with $\norm{z}=d(0,C)$. By convexity,
$\norm{w}^2\geq \norm{z}^2+\norm{w-z}^2$ for every $w\in C$. By the triangle inequality,
\[
d(z,C^+)+d(z,C^-)>2\delta.
\]
So for some $\sigma\in\{+,-\}$ we have $d(z,C^\sigma)>\delta$, and therefore
\[
d(0,C^\sigma)^2\geq d(0,C)^2+\delta^2.
\]
Since $|V^\sigma| \ge (\eta/2)|V|$, the bound \eqref{eq:littlestoneinductivehypothesis} now follows by induction.
\end{proof}

We now finish by combining \Cref{obsYprime}, \Cref{obsvprime} and \Cref{prop:etadeltaconstantsubset}.

\begin{proof}[Proof of \Cref{lemdecomposition}]
    Apply \Cref{obsYprime} to get a set $Y'$, and apply \Cref{prop:etadeltaconstantsubset} with parameters $\eta = \eps/(2\gamma + 1/2)$ and $\delta = 1/4$ to the set of vectors indexed by $Y'$ to get some $S \subseteq Y'$. Since $|Y'| \ge 4(6\gamma/\eps)^{16\gamma^2}$, we have $|S| \ge 2$, so we may apply \Cref{obsvprime} to get upper and lower bounds on $\norm{v'}^2$.

    Since $\ex_{y \in S} \langle v', v_y \rangle = \langle v', v'\rangle$, there exists $y_0 \in S$ such that $\langle v', v_{y_0} - v' \rangle \le 0$. We now use the decomposition $v_{y_0} = v' + v''$, where $v'' = v_{y_0} - v'$. Since $\langle v', v'' \rangle \le 0$, we have $\norm{v'}^2 + \norm{v''}^2 \le \norm{v_{y_0}}^2$. Property (i) of the decomposition in \Cref{lemdecomposition} now follows from the upper and lower bounds on $\norm{v'}^2$ given by \Cref{obsvprime}.

    It now remains to check property (ii). For $x \in X$, since $S$ is $(\eta, \delta)$-constant, there exists some $g \in \RR$ such that
    \[ \Pr_{y \in S}(|\langle u_x, v_y \rangle - g| > \delta) \le \eta .\]
    Let $g' \in \ZZ$ be the result of rounding $g$ to the nearest integer. Since $A$ is $\eps$-almost integer-valued, we have that $|A(x, y) - g'| \le \eps$ holds whenever $|A(x, y) - g| \le 1/4$. Therefore
    \[ |\langle u_x, v' \rangle - g'| \le \ex_{y \in S} |\langle u_x, v_y \rangle - g'| \le (1 - \eta)\eps + \eta(2\gamma + 1/2) \le 2\eps ,\]
    so the inner product $\langle u_x, v' \rangle$ is at distance at most $2\eps$ from an integer. Since $A$ is $\eps$-almost integer-valued and $\langle u_x, v'' \rangle = \langle u_x, v_{y_0} \rangle - \langle u_x, v'' \rangle$, we get that $\langle u_x, v'' \rangle$ is at distance at most $3\eps$ from an integer.
\end{proof}
 
\section{Measurable Schur multipliers} 
\label{sec:measurable}

We conclude by extending \Cref{thmmain} to Schur multipliers over measure
spaces. Let $(X,\mu)$ and $(Y,\nu)$ be standard $\sigma$-finite measure
spaces. Recall that every $f\in L^2(X\times Y)$ defines a Hilbert--Schmidt operator $T_f:L^2(X)\to L^2(Y)$ by
\[
(T_f\xi)(y)=\int_X f(x,y)\xi(x)\,d\mu(x).
\]
A function $P\in L^\infty(X\times Y)$ is a \emph{measurable Schur
multiplier} if there is a constant $C$ such that
$
\norm{T_{Pf}}\le C\norm{T_f}
$
for every $f\in L^2(X\times Y)$; the least such $C$ is the {\it Schur
multiplier norm} $\normmul P$ of $P$. (See \cite[Section~2]{MR2777487} for further exposition.)

A boolean measurable function $B:X\times Y\to\{0,1\}$ is called
\emph{blocky} if, up to a null set,
\[
\supp(B)=\bigcup_{k\in K}X_k\times Y_k,
\]
where $K$ is countable and $\{X_k\}_{k\in K}$ and $\{Y_k\}_{k\in K}$
are families of pairwise disjoint measurable sets. As in the discrete
case, every blocky measurable Schur multiplier is a contractive Schur multiplier: the
corresponding map on operators is a direct sum of compressions to
$L^2(X_k)$ and $L^2(Y_k)$, and hence has norm at most $1$.

We prove that, just as in the countable case, every idempotent measurable Schur
multiplier is a finite signed sum of contractive idempotent measurable
Schur multipliers. 

\begin{theorem}
\label{thm:measurable}
There exists an absolute constant $C$ such that the following holds.
Let $(X,\mu)$ and $(Y,\nu)$ be standard $\sigma$-finite measure spaces,
and let $P\in L^\infty(X\times Y)$ be an idempotent measurable Schur
multiplier with $\normmul P \le \gamma$. Then there exists $L\le 2^{C\gamma^4}$ such that 
\[
P=\sum_{i=1}^L \sigma_i B_i,
\]
almost everywhere, for blocky measurable Schur multipliers $B_i$ and signs $\sigma_i \in \{-1,1\}$.
\end{theorem}
 
\begin{proof}
Since $P(x,y)^2=P(x,y)$ almost everywhere, after modifying $P$ on a null set we may assume $P:X\times Y\to\{0,1\}$. By the factorization theorem for measurable Schur multipliers~\cite[Theorem~2.1]{MR2777487}, there are measurable maps $u:X\to\ell^2$ and $v:Y\to\ell^2$ such that
\[
P(x,y)=\langle u(x),v(y)\rangle,\qquad
\norm{u(x)}\le1,\qquad \hbox{and}\qquad\norm{v(y)}\le\gamma
\]
almost everywhere. Discarding null subsets, we may assume that for every $x$ this identity holds for almost every $y$, and vice versa.

We first show that if rows that agree almost everywhere are identified with one another, then there are only countably many distinct rows. Choose one $x$ from
each row class. If $x,x'$ have distinct rows, then they differ on a set
of positive measure, and hence for some $y$,
\[
1=\Bigl|\bigl\langle u(x)-u(x'),v(y)\bigr\rangle\Bigr|
   \le \gamma\bignorm{u(x)-u(x')}.
\]
Thus the corresponding vectors $u(x)$ are $(1/\gamma)$-separated, and
so there are only countably many row classes by separability of $\ell^2$.

The row classes are measurable by an argument of \cite{MR365167} (see also~\cite[Theorem~7]{kp05}).  To see this, fix a representative $x_i$. The set
\[ \bigl\{(a,b)\in\ell^2\times\ell^2:\langle a,b\rangle\ne0\bigr\} \]
is open, and since $\ell^2\times\ell^2$ has a countable basis of open rectangles, it is a countable union of such rectangles. Pulling back along the measurable map $(x,y)\mapsto\bigl(u(x)-u(x_i),\,v(y)\bigr)$, the set
\[ \Bigl\{(x,y):\bigl\langle u(x)-u(x_i),v(y)\bigr\rangle\ne0\Bigr\} \]
is a countable union $\bigcup_n A_n\times B_n$ of measurable rectangles. Now $x$ lies in the row class of $x_i$ precisely when this set meets $\{x\}\times Y$ in a null set, that is, when $x\notin A_n$ for every $n$ with $\nu(B_n)>0$. Hence, up to a null set,
$ X_i=X\setminus\bigcup_{n:\,\nu(B_n)>0}A_n$, 
and is therefore measurable. Interchanging the roles of $X$ and $Y$ yields countably many measurable column classes $\{Y_j\}_{j\in J}$.

We therefore obtain countable measurable partitions 
$ X=\bigcup_{i\in I}X_i$  and $Y=\bigcup_{j\in J}Y_j$,
modulo null sets, such that $P$ is constant almost everywhere on each $X_i\times Y_j$. Let $A:I\times J\to\{0,1\}$
denote the resulting countable boolean matrix. Choosing representatives $x_i\in X_i$ and $y_j\in Y_j$ outside the relevant null sets, we have 
\[ A(i,j)=\bigl\langle u(x_i),v(y_j)\bigr\rangle, \] 
with $\bignorm{u(x_i)}\le1$ and $\bignorm{v(y_j)}\le\gamma$. Hence $\normgamma A\le\gamma$.
Applying \Cref{thmmain} to the countable boolean matrix $A$ yields a representation $A=\sum_{i=1}^L\sigma_i B_i$ with $L\le 2^{C\gamma^4}$, where each $B_i$ is blocky. Each $B_i$ pulls back along the partitions $\{X_i\}_{i\in I}$ and $\{Y_j\}_{j\in J}$ to a blocky measurable Schur multiplier, and the resulting identity holds almost everywhere.  
\end{proof}

\section*{Acknowledgements}
The first, fourth, and fifth listed authors would like to thank Julian Sahasrabudhe for suggesting the problem and for helpful discussions. The second and third authors are supported by the Natural Sciences and Engineering Research Council of Canada. The fourth author is supported by a {\small DA EPSRC} Heilbronn Institute for
Mathematical Research doctoral training 2025 grant (no.~{\small UKRI}3009).

\bibliographystyle{alphacitation}
\xpatchcmd{\em}{\itshape}{\slshape}{}{}
\bibliography{citations}

\bigskip\goodbreak\noindent
\textsc{Department of Pure Mathematics and Mathematical Statistics, University of Cambridge,  Cambridge, United Kingdom}

\smallskip\noindent
\textsl{E-mail address}: \texttt{cb2138@cam.ac.uk}

\bigskip\goodbreak\noindent
\textsc{Department of Mathematics and Statistics, McGill University, Montreal, Quebec, Canada}

\smallskip\noindent
\textsl{E-mail address}: \texttt{marcel.goh@mail.mcgill.ca}

\bigskip\goodbreak\noindent
\textsc{School of Computer Science, McGill University, Montreal, Quebec, Canada}

\smallskip\noindent
\textsl{E-mail address}: \texttt{hatami@cs.mcgill.ca}

\bigskip\goodbreak\noindent
\textsc{Department of Pure Mathematics and Mathematical Statistics, University of Cambridge,  Cambridge, United Kingdom}

\smallskip\noindent
\textsl{E-mail address}: \texttt{scj47@cam.ac.uk}

\bigskip\goodbreak\noindent
\textsc{Department of Pure Mathematics and Mathematical Statistics, University of Cambridge,  Cambridge, United Kingdom}

\smallskip\noindent
\textsl{E-mail address}: \texttt{dn410@cam.ac.uk}

\end{document}